\documentclass[11pt]{article}
\usepackage{amssymb,epsfig}
\usepackage{amsmath,amsthm}
\usepackage{amscd}
\usepackage{graphicx}
\usepackage{epstopdf}
\usepackage{color}
\usepackage{multirow,array}

\title {\bf A nonlinear PPH-type reconstruction based on equilateral triangles.}

\author{S. Amat\thanks{Departamento de Matem\'atica Aplicada y Estad\'{\i}stica.
 Universidad Polit\'ecnica de Cartagena (Spain).
 e-mail:{\tt sergio.amat@upct.es.} The first four authors have been supported through the Proyecto financiado
por la Comunidad Aut\'onoma de la Regi\'on de Murcia a trav\'es de la convocatoria de Ayudas a proyectos
para el desarrollo de investigaci\'on cient\'{\i}fica y t\'ecnica por grupos competitivos, incluida en el Programa
Regional de Fomento de la Investigaci\'on Cient\'{\i}fica y T\'ecnica (Plan de Actuaci\'on 2018) de la Fundaci\'on
S\'eneca-Agencia de Ciencia y Tecnolog\'{\i}a de la Regi\'on de Murcia 20928/PI/18 and by the Spanish national research project PID2019-108336GB-I00.} \and P. Ortiz\thanks{
 Departamento de Matem\'atica Aplicada y Estad\'{\i}stica.
 Universidad Polit\'ecnica de Cartagena (Spain).
 e-mail:{\tt portiz@navantia.es.}} \and
 J. Ruiz\thanks{
 Departamento de Matem\'atica Aplicada y Estad\'{\i}stica.
 Universidad Polit\'ecnica de Cartagena (Spain).
e-mail:{\tt juan.ruiz@upct.es.}}
 \and
 J.C.Trillo\thanks{
 Departamento de Matem\'atica Aplicada y Estad\'{\i}stica.
 Universidad Polit\'ecnica de Cartagena (Spain).
e-mail:{\tt jc.trillo@upct.es.}} \and
D. F. Ya\~nez \thanks{
 Departamento de Matem\'aticas.
 Universidad de Valencia (Spain).
e-mail:{\tt dionisio.yanez@uv.es.}}
}

\begin{document}

\maketitle

\newtheorem{proposition}{Proposition}
\newtheorem{lemma}{Lemma}
\newtheorem{definition}{Definition}
\newtheorem{theorem}{Theorem}
\newtheorem{corollary}{Corollary}
\newtheorem{remark}{Remark}

\begin{abstract}
In this paper we introduce a new nonlinear reconstruction operator over two dimensional triangularized domains with equilateral triangles. We focus on the local definition of the operator.
The ideas behind this definition come from some basic properties of the Harmonic mean of three positive values. We prove some results regarding the approximation properties of the operator and we carry out
some numerical tests giving evidence of the avoidance of any Gibbs effects.
\end{abstract}


\vspace{10pt} {\bf AMS(MOS) subject classifications.} 41A05, 41A10, 65D17.

\section{Introduction}\label{sec1}

The arithmetic and the harmonic mean of positive numbers are present in many scientific applications ranging from statistics to numerical analysis. The harmonic mean has the property of penalizing large
values, giving rise, because of this reason, to several interesting applications. Moreover, when the arguments do not differ much from each other, both means remain close, which is another crucial property in
applications.

In our field of research both the arithmetic mean and the harmonic mean have been used successfully in several occasions for different applications. See for instance \cite{Susa,Mar} for an example in numerical conservation laws, \cite{ADLT,ADLT02,AL04,Trillo_thesis} for applications regarding signal processing and signal compression, \cite{ALRT,Den1} for their use in image denoising and compression, and \cite{ADT,ACRT,KD} for the case of generation of curves and subdivision.

In \cite{OT3} a nonlinear reconstruction operator called PPH (Piecewise Polynomial Harmonic) was extended to nonuniform grids by using a specific weighted harmonic mean instead of the standard harmonic mean. In this paper our aim is to give rise to non separable reconstructions in two dimensions based on similar ideas as the ones used to build the original PPH reconstruction \cite{ADLT} by using two main basic properties of the harmonic mean of two positive values, as it has been mentioned above.
More specifically speaking, we need to dispose of an appropriate mean in two dimensions which satisfies these required basic properties as the harmonic mean does. In fact, the straightforward choice of using the harmonic mean of three positive values as a candidate works as it will be shown through this paper.

Nonlinear means appear as good candidates to define adapted reconstruction methods which mi\-nimize the undesirable effects provoked by the presence of a discontinuity in the data. In fact, we will show in the numerical experiments section one particular new reconstruction method for two dimensional functions which seems to avoid the Gibbs effects according to the numerical examples. This re\-construction extends somehow the PPH reconstruction defined in \cite{ADLT,OT3}, which was proven theoretically to avoid the mentioned Gibbs effects.

The paper is organized as follows: In Section \ref{sec2} we work with the arithmetic and harmonic means of three positive numbers, proving two essential results about these means which will allow us to define
 adapted reconstruction operators in the numerical experiments section.  In Section \ref{sec3} we explicitly define a new reconstruction in $2D$ over triangular meshes adapted to discontinuities.  Finally, in Section \ref{sec4} we give some perspectives and conclusions.

\section{Harmonic mean of three positive values} \label{sec2}
In this section we present the properties about the harmonic mean of three positive values that are relevant for the rest of the paper.
\begin{definition}\label{M3}
Given $a_1>0,$ $a_2>0,$ $a_3>0$  three positive real numbers, their harmonic mean is defined by
\begin{equation*} \label{eq:HW3}
	H_3(a_1,a_2,a_3)= \dfrac{3a_1 a_2 a_3}{a_2 a_3 + a_1 a_3 + a_1 a_2}.
\end{equation*}
\end{definition}

\begin{lemma} \label{lemaAcotacionMinimo3}
If $a_1 > 0,$ $a_2 > 0,$ $a_3 >0,$ the harmonic mean is bounded as follows
	\begin{equation} \label{eq:boundedVnu3}
		H_3(a_1,a_2,a_3) < 3 \min\left\{ a_1,a_2,a_3\right\}.
	\end{equation}
\end{lemma}

\begin{proof}

\begin{equation*}
	H_{3}(a_1,a_2,a_3)= \dfrac{3a_{1}a_{2}a_{3}}{a_2 a_3 + a_1 a_3 + a_1 a_2} \leq 3 a_{i}, \quad i=1,2,3.
\end{equation*}
\end{proof}

\begin{lemma}\label{lemaOrdenMedias3}
Let $a>0$ a fixed positive real number, and let $a_1\geq a,$ $a_2\geq a,$ $a_3\geq a.$ If  $|a_1-a_2|= O(h),$ $|a_1-a_3|= O(h),$  then the harmonic mean is also
close to the arithmetic mean $M_3(a_1,a_2,a_3)=\dfrac{a_1+a_2+a_3}{3},$
	\begin{eqnarray}\label{eq:distMVnu3} \notag
		|M_3(a_1, a_2,a_3)-H_3(a_1,a_2,a_3)|&=& \dfrac{(a_1-a_2)^2 a_3+(a_1-a_3)^2 a_2+(a_2-a_3)^2 a_1 }{3(a_2 a_3 + a_1 a_3 + a_1 a_2)}\\
&=& O(h^2).
	\end{eqnarray}
\end{lemma}

\begin{proof}
\begin{eqnarray*}
	|M_3(a_1,a_2,a_3)-H_{3}(a_1,a_2,a_3)|&=& |\dfrac{a_1+a_2+a_3}{3} - \dfrac{3a_{1}a_{2}a_{3}}{a_2 a_3 + a_1 a_3 + a_1 a_2}|\\
&=& |\dfrac{(a_1-a_2)^2 a_3+(a_1-a_3)^2 a_2+(a_2-a_3)^2 a_1 }{3(a_2 a_3 + a_1 a_3 + a_1 a_2)}|.
\end{eqnarray*}
Since $|a_1-a_2|=O(h),$ $|a_1-a_3|=O(h),$ and $|a_2-a_3|\leq |a_2-a_1|+|a_1-a_3|=O(h)+O(h)=O(h),$ we directly get the result.
\end{proof}

\section{A non separable PPH type local reconstruction operator over equilateral triangles} \label{sec3}
In this section our purpose is to define a nonlinear reconstruction operator adapted to jump discontinuities. In what follows, we are going to present a new nonlinear adapted recons\-truction
method for approximating two variable functions using the point values of the function over triangular meshes.
We are going to focuss on the local definition of the reconstruction operator for a given triangle of the mesh.
Let us consider $S \subseteq \mathbb{R}^2$ the equilateral triangle with sides of length $2h,$ with $h>0$ any positive real number, defined by the vertices $A (-h,\frac{\sqrt{3}}{2}h),$
$C (0,-\frac{\sqrt{3}}{2}h),$ $E (h,\frac{\sqrt{3}}{2}h),$ as shown in Figure \ref{fig:trian1S}. Let us also consider that the triangle is divided into $4$ new smaller triangles: $S_A$ of vertices $ABF,$
$S_C$ of vertices $CDB,$ $S_E$ of vertices $EFD,$ and $S_{R}$ of vertices $BDF$, just by considering the mid points of each side of the original triangle, see also Figure \ref{fig:tria3}. We are going to describe how to build a nonlinear reconstruction inside
the triangle $S_R$ of an underlying function $f(x,y),$ from which we know its point values at the six mentioned points $A,B,C,D,E,F.$ This nonlinear reconstruction will attain third order of approximation in case the underlying function $f(x,y)$ is of class $C^3,$ and will be adapted to the presence of jump discontinuities that affect only one of the three values $A,$ $C,$ or $E,$ see also Figure \ref{fig:trian2}.\\
\begin{figure}
\begin{center}
  \includegraphics[width=6.5cm]{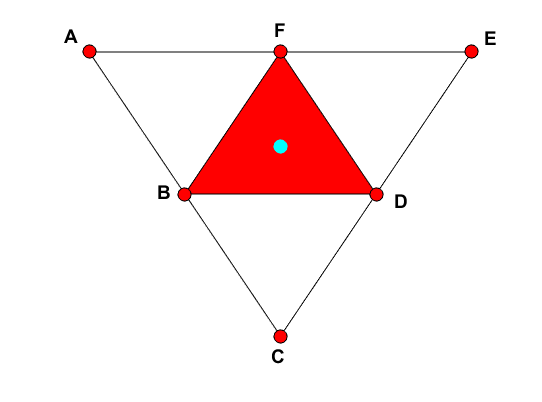}\\
\end{center}
\caption{Disposition of the considered domain to build the reconstruction inside the red triangle $S_R$ with vertices $BDF,$ using the point values of an underlying function $f(x,y)$ at the six points
$A,B,C,D,E,F.$}\label{fig:trian1S}
\end{figure}
Firstly, we are going to define the associated linear reconstruction, that it is going to be nothing more than the second degree interpolating polynomial that goes through the six given initial points.
Let us write the polynomial around the barycenter of the triangle $G (0,\frac{\sqrt{3}}{6}h)$ in the form
\begin{eqnarray} \label{pollin}
p(x,y)&=& a_{00}+a_{10}x+a_{01}(y-\frac{\sqrt{3}}{6}h)+a_{20}x^2+a_{11}x(y-\frac{\sqrt{3}}{6}h)+a_{02}(y-\frac{\sqrt{3}}{6}h)^2.
\end{eqnarray}
Imposing the interpolation conditions $p(Q_i)=f(Q_i),$ for $Q_i \in \{A,B,C,D,E,F\},$ we get a linear system of equations, which has unique solution given by
\begin{eqnarray} \label{coeflin}
a_{00}&=& \frac{4}{9}(f_B+f_D+f_F)-\frac{1}{9}(f_A+f_C+f_E), \\ \notag
a_{10}&=& \frac{-f_A-4f_B+4f_D+f_E}{6h},\\ \notag
a_{01}&=& \frac{\sqrt{3}}{18h}(f_A-4f_B-2f_C-4f_D+f_E+8f_F),\\ \notag
a_{20}&=& \frac{f_A-2f_F+f_E}{2h^2},\\ \notag
a_{11}&=& -\frac{\sqrt{3}}{3h^2}(f_A-2f_B+2f_D-f_E),\\ \notag
a_{02}&=& \frac{f_A-4f_B+4f_C-4f_D+f_E+2f_F}{6h^2}, \notag
\end{eqnarray}
where $f_{Q_i}$ denotes $f(Q_i).$
It is easy to prove, by using Taylor expansions, the following theorem that ensures third order of approximation of the proposed linear reconstruction.
\begin{theorem} \label{teo:recolin}
Let $f:\Omega \Rightarrow \mathbb{R}$ be a function of class $C^3(\Omega),$ with $S \subseteq \Omega.$ And let $p(x,y)$ denote the interpolating polynomial defined by (\ref{pollin}) with the coefficients
 given by (\ref{coeflin}). Then, we have
\begin{equation*}
|f(x,y) - p(x,y)|=O(h^3), \ \forall \ (x,y) \in S_R.
\end{equation*}
\end{theorem}
We have then accomplished the first step in the definition of the nonlinear method, that is, we have a ready to modify linear method.
Secondly, we want to rewrite the coefficients of the linear reconstruction by making appear arithmetic means. Let us define $\Delta_A,$ $\Delta_C,$ and $\Delta_E$ as follows
\begin{eqnarray*} \label{def:difdiv2d}
\Delta_{A}&:=&\frac{f_A-(f_B+f_F)+f_D}{h},\\ \notag
\Delta_{C}&:=&\frac{f_C-(f_B+f_D)+f_F}{h},\\ \notag
\Delta_{E}&:=&\frac{f_E-(f_D+f_F)+f_B}{h}.\notag
\end{eqnarray*}
It is immediate to prove, by using Taylor expansions, that in smooth areas of the function
\begin{eqnarray*} \label{prop:difdiv2d}
&& \Delta_A=O(h), \ \Delta_C=O(h), \ \Delta_E=O(h),\\
&& \Delta_{A}-\Delta_{C}=O(h), \ \Delta_{A}-\Delta_{E}=O(h), \ \Delta_{C}-\Delta_{E}=O(h).
\end{eqnarray*}
Moreover, these values $\Delta_A,$ $\Delta_C,$ $\Delta_E$ act as smoothness indicators, a kind of divided differences, in the sense that if a jump discontinuity lies affecting one of the values $A,$ $C,$ or $E,$
then the corresponding divided difference will be $O(\frac{1}{h}),$ while the others will remain $O(h).$
In Figure \ref{fig:trian2}, we see the case of having the vertex $E$ affected by a jump discontinuity, which takes place along a curve plotted in blue. The idea behind the method that we are going to explain is
to substitute $f_E$ for a more suitable value $\widetilde{f}_E,$ that both maintains the approximation accuracy in case of dealing with a smooth function and allows for adaptation in case of discontinuity.\\
\noindent The coefficients in (\ref{coeflin}) can be rewritten as follows
\begin{eqnarray}\label{coeflinMean}
a_{00}&=& \frac{1}{3}(f_B+f_D+f_F)-\frac{h}{3}\frac{\Delta_{A}+\Delta_{C}+\Delta_{E}}{3}, \\ \notag
a_{10}&=& \frac{f_D-f_B}{h}-\frac{1}{6}(2\Delta_{A}+\Delta_{C})+\frac{1}{2}\frac{\Delta_{A}+\Delta_{C}+\Delta_{E}}{3},\\  \notag
a_{01}&=& \frac{\sqrt{3}}{6}(\Delta_{C}+2\frac{f_F-f_C}{h})+\frac{\sqrt{3}}{6}\frac{\Delta_{A}+\Delta_{C}+\Delta_{E}}{3},\\ \notag
a_{20}&=& -\frac{3}{2h}\Delta_{C}+\frac{3}{2h}\frac{\Delta_{A}+\Delta_{C}+\Delta_{E}}{3},\\ \notag
a_{11}&=& -\frac{2\sqrt{3}}{h}\Delta_{A}+\frac{\sqrt{3}}{3h^2}(f_D-2f_F+f_B)+ \frac{\sqrt{3}}{h}\frac{\Delta_{A}+\Delta_{C}+\Delta_{E}}{3},\\ \notag
a_{02}&=& -\frac{1}{2h}\Delta_{C}+\frac{1}{2h}\frac{\Delta_{A}+\Delta_{C}+\Delta_{E}}{3}.\\ \notag
\end{eqnarray}
\begin{figure}
\begin{center}
  \includegraphics[width=6.5cm]{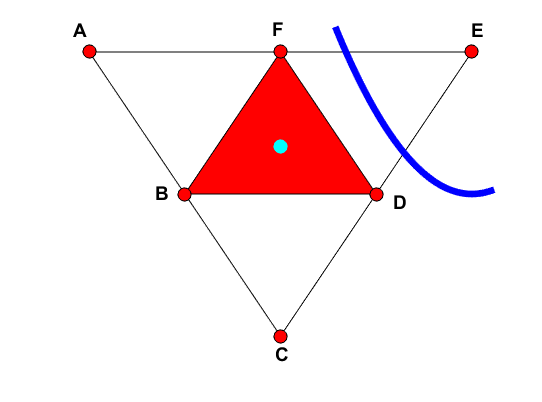}\\
\end{center}
\caption{Disposition of the considered domain affected by a jump discontinuity along the blue curve.} \label{fig:trian2}
\end{figure}
It is important that the potentially affected value by a possible discontinuity $f_E,$ the one that makes $\Delta_E$ be the largest in absolute value, appears only inside the term $\Delta_E$ and
in turn $\Delta_E$ appears only in the arithmetic mean.\\

Thirdly, we are going to modify the expressions of the coefficients in (\ref{coeflinMean}) implementing the substitution of the arithmetic means by adequate harmonic means.
Since the values of the divided differences could be positive, negative, or zero, and we are also going to need that these quantities satisfy the hypothesis of Lemmas \ref{lemaAcotacionMinimo3} and \ref{lemaOrdenMedias3},
we require the redefinition of the harmonic mean by using a translation strategy. In order to do so, we introduce
the concept of translation operator, which will allow us to extend the definition of the harmonic mean.
\begin{definition} \label{defTras}
Given $h>0,$ a translation operator $T$ is any function $T:\mathbb{R}^3 \rightarrow \mathbb{R}$ satisfying
\begin{enumerate}
\item $T(0,0,0)=0,$
\item $T(x,y,z)=T(\sigma(x),\sigma(y),\sigma(z)),$ where $\sigma$ is any permutation of three elements,
\item $T(-x,-y,-z)=-T(x,y,z),$
\item $sign(x+T(x,y,z))=sign(y+T(x,y,z))=sign(z+T(x,y,z)),$ $\forall \ (x,y,z)\neq (0,0,0),$
\item if $(x,y,z)\neq (0,0,0),$ with $|s|=\max\{|x|,|y|,|z|\},$
\begin{itemize}
\item[a)] if $\exists \ s_1: |s_1|=|s|, sign(s_1)\neq sign(s),$ then $sign(x+T(x,y,z))>0,$ \\ $sign(y+T(x,y,z))>0,$ $sign(z+T(x,y,z))>0,$
\item[b)] if $\nexists \ s_1: |s_1|=|s|, sign(s_1)\neq sign(s),$ then $sign(x+T(x,y,z))sign(s)>0,$ \\ $sign(y+T(x,y,z))sign(s)>0,$ $sign(z+T(x,y,z))sign(s)>0,$
\end{itemize}
\item $\min\{|x+T(x,y,z)|,|y+T(x,y,z),|z+T(x,y,z)|\}=O(1),$  $\forall \ (x,y,z)\neq (0,0,0),$ with $|x|=O(h^{\alpha}),$ $|y|=O(h^{\alpha}),$ $|z|=O(h^{\alpha}),$ for some $\alpha \geq 0.$
\end{enumerate}
\end{definition}
Properties $1$ to $4$ are meant to apply the harmonic mean in mind by using basically the expression given for positive numbers. While the property $5$ will play an important role to guarantee the
adaptation of the method in case one of the arguments is very large due to the presence of a discontinuity. In turn, property $6$ ensures that the new arguments that are going to be considered in the new definition
of the mean will satisfy the hypothesis of Lemma \ref{lemaOrdenMedias3}.

We are now ready to redefine the weighted harmonic mean
\begin{equation} \label{meanJ}
J_3(a_1,a_2,a_3)=\left\{ \begin{array}{ll}
H_3(a_1+T,a_2+T,a_3+T)-T, & (a_1,a_2,a_3)\neq (0,0,0),\\
0, & (a_1,a_2,a_3)= (0,0,0),\\
\end{array} \right.
\end{equation}
where $T$ is any translation operator satisfying Definition \ref{defTras}. It it important to notice that the new mean also satisfy similar lemmas, Lemma \ref{lemaAcotacionMinimo3} and Lemma \ref{lemaOrdenMedias3},
as the harmonic mean. In fact, we can prove the following two lemmas.
\begin{lemma} \label{minimo3T}
Let $a_i>0, \ i=1,2,3$ be be real numbers. Then, the translated harmonic mean $J_3$ is bounded as follows
\begin{equation*}
	|J_{3}(a_1,a_2,a_3)| \leq \max\{ 3 |a_1+T|, |T|\}.
\end{equation*}
\end{lemma}

\begin{proof}
Since $J_3(a_1 + T, a_2+T,a_3+T)$ and $T$ have the same sign, then applying Lemma \ref{lemaAcotacionMinimo3} we get
\begin{eqnarray*}
		|J_3(a_1,a_2,a_3)| &\leq& \max\left\{|H_3(a_1 + T, a_2+T,a_3+T)|,|T|\right\}\leq \max\left\{3|a_1+T|,|T|\right\}.\\
	\end{eqnarray*}
\end{proof}

\begin{lemma} \label{orden3T}
Let $a_i>0, \ i=1,2,3$ be real numbers.
If $\left|a_{1}-a_{i}\right| = O(h), \ i=2,3,$ then, the translated weighted harmonic mean $J_3$ and the arithmetic
mean $M_3:=\dfrac{a_1+a_2+a_3}{3}$ satisfy
\begin{equation*}
\left|M_{3} - J_{3} \right|=O(h^{2}).
\end{equation*}
\end{lemma}

\begin{proof}
The case of $(a_1,a_2,a_3)=(0,0,0)$ is trivial. If $(a_1,a_2,a_3)\neq (0,0,0),$
using the definition of $J_3$ we get
\begin{eqnarray*}
		|M_w(a_1,a_2,a_3)-J_3(a_1,a_2,a_3)|&=& |M_3(a_1,a_2,a_3)-H_3(a_1+T,a_2+T,a_3+T)-T|\\
&=&|M_3(a_1+T,a_2+T,a_3+T)-H_3(a_1+T,a_2+T,a_3+T)|,
	\end{eqnarray*}
and applying Lemma \ref{lemaOrdenMedias3} we have that
\begin{eqnarray*}
|M_3(a_1+T,a_2+T,a_3+T)-H_3(a_1+T,a_2+T,a_3+T)|&=&O(h^2).
\end{eqnarray*}
\end{proof}

Thanks to the new translated version of the weighted harmonic mean in (\ref{meanJ}) we can finally define the modified coefficients
\begin{eqnarray}\label{coefPPHtria}
\widetilde{a}_{00}&=& \frac{1}{3}(f_B+f_D+f_F)-\frac{h}{3}J_{3}(\Delta_{A},\Delta_{C},\Delta_{E}), \\ \notag
\widetilde{a}_{10}&=& \frac{f_D-f_B}{h}-\frac{1}{6}(2\Delta_{A}+\Delta_{C})+\frac{1}{2}J_{3}(\Delta_{A},\Delta_{C},\Delta_{E}),\\  \notag
\widetilde{a}_{01}&=& \frac{\sqrt{3}}{6}(\Delta_{C}+2\frac{f_F-f_C}{h})+\frac{\sqrt{3}}{6}J_{3}(\Delta_{A},\Delta_{C},\Delta_{E}),\\ \notag
\widetilde{a}_{20}&=& -\frac{3}{2h}\Delta_{C}+\frac{3}{2h}J_{3}(\Delta_{A},\Delta_{C},\Delta_{E}),\\ \notag
\widetilde{a}_{11}&=& -\frac{2\sqrt{3}}{h}\Delta_{A}+\frac{\sqrt{3}}{3h^2}(f_D-2f_F+f_B)+ \frac{\sqrt{3}}{h}J_{3}(\Delta_{A},\Delta_{C},\Delta_{E}),\\ \notag
\widetilde{a}_{02}&=& -\frac{1}{2h}\Delta_{C}+\frac{1}{2h}J_{3}(\Delta_{A},\Delta_{C},\Delta_{E}).\\ \notag
\end{eqnarray}
The new nonlinear local reconstruction method writes then
\begin{eqnarray} \label{polPPH}
\widetilde{p}(x,y)&=& \widetilde{a}_{00}+\widetilde{a}_{10}x+\widetilde{a}_{01}(y-\frac{\sqrt{3}}{6}h)+\widetilde{a}_{20}x^2+\widetilde{a}_{11}x(y-\frac{\sqrt{3}}{6}h)+\widetilde{a}_{02}(y-\frac{\sqrt{3}}{6}h)^2,
\end{eqnarray}
where the coefficients $\widetilde{a}_{00},$ $\widetilde{a}_{10},$ $\widetilde{a}_{01},$ $\widetilde{a}_{20},$ $\widetilde{a}_{11},$ $\widetilde{a}_{02}$ are given in (\ref{coefPPHtria}).
It is also interesting to notice that this reconstruction amounts to modifying the value $f_E$
\begin{eqnarray*}
f_E=f_B+f_D+f_F-(f_A+f_C)+3h M_{3}(\Delta_{A},\Delta_{C},\Delta_{E}),\\
\end{eqnarray*}
in order to get
\begin{eqnarray*}
\widetilde{f}_E=f_B+f_D+f_F-(f_A+f_C)+3h J_{3}(\Delta_{A},\Delta_{C},\Delta_{E}),\\
\end{eqnarray*}
and then considering the original interpolation problem with modified function values $\{f_A,f_B,f_C,f_D,f_F,\widetilde{f}_E\}.$
By definition, it is not difficult to prove a theorem about the adaptation of the proposed method and the third order accuracy in smooth areas.
\begin{theorem} \label{teo:recoPPH}
Let $f:\Omega \Rightarrow \mathbb{R}$ be a function of class $C^3(\Omega),$ with $S \subseteq \Omega.$ And let $\widetilde{p}(x,y)$ denote the interpolating polynomial defined by (\ref{polPPH}) with the coefficients
 given by (\ref{coefPPHtria}). Then, we have
\begin{equation} \label{ordenptildeSmooth}
|f(x,y) - \widetilde{p}(x,y)|=O(h^3), \ \forall \ (x,y) \in S_R.
\end{equation}
Moreover, if $f$ has a jump discontinuity along a curve letting $S\setminus S_E$ to one side and the vertex $E$ to the other side of the curve, then we have
\begin{equation} \label{ordenptildeDis}
|f(x,y) - \widetilde{p}(x,y)|=O(h), \ \forall \ (x,y) \in S_R.
\end{equation}
\end{theorem}
\begin{proof}
Taking into account that $|M_{3}-J_{3}|=O(h^2)$ according to Lemma \ref{orden3T}, from (\ref{coeflinMean}) and (\ref{coefPPHtria}) we get that
\begin{equation}\label{difcoef}
\begin{array}{lll}
a_{00} - \widetilde{a}_{00}= O(h^3), & & \\
a_{10} - \widetilde{a}_{10}=O(h^2), & a_{01} - \widetilde{a}_{01}= O(h^2), & \\
a_{20} - \widetilde{a}_{20}=O(h), & a_{11} - \widetilde{a}_{11}= O(h), & a_{02} - \widetilde{a}_{02}= O(h).\\
\end{array}
\end{equation}
Now, from the expressions of the linear reconstruction $p(x,y)$ in (\ref{pollin}) and of the nonlinear reconstruction $\widetilde{p}(x,y)$ in (\ref{polPPH}) we easily obtain by applying
the triangular inequality that
\begin{eqnarray} \label{difpol}\notag
|p(x,y)-\widetilde{p}(x,y)|&\leq& |a_{00} -\widetilde{a}_{00}|+|a_{10} -\widetilde{a}_{10}| |x|+ |a_{01} -\widetilde{a}_{01}| |y-\frac{\sqrt{3}}{6}h|+ |a_{20} -\widetilde{a}_{20}||x|^2\\
&+& |a_{11} -\widetilde{a}_{11}||x||y-\frac{\sqrt{3}}{6}h|+ |a_{02} -\widetilde{a}_{02}||y-\frac{\sqrt{3}}{6}h|^2.
\end{eqnarray}
Thus, using (\ref{difcoef}) we reach to
\begin{equation} \label{difpol1}
|p(x,y)-\widetilde{p}(x,y)|=O(h^3).
\end{equation}
Applying Theorem \ref{teo:recolin} and (\ref{difpol1}) we have
\begin{equation*}
|f(x,y)-\widetilde{p}(x,y)|\leq |f(x,y)-p(x,y)|+|p(x,y)-\widetilde{p}(x,y)|=O(h^3),
\end{equation*}
which proves (\ref{ordenptildeSmooth}).\\
In order to prove (\ref{ordenptildeDis}) we start by pointing out that
\begin{equation}\label{teo:p1}
|f(x,y)-p_1(x,y)|=O(h^2), \ \forall (x,y) \in S_R,
\end{equation}
where $p_1(x,y)$ is given by
\begin{equation*}
p_1(x,y)=a_{00}+a_{10}x+a_{01}(y-\frac{\sqrt{3}}{6}h).
\end{equation*}
Now, taking into account that due to Lemma \ref{minimo3T}, $\ |J_{3}|=O(1),$ we have
\begin{eqnarray} \label{difpolp1}
|p_1(x,y)-\widetilde{p}(x,y)|&\leq& |a_{00} -\widetilde{a}_{00}|+|a_{10} -\widetilde{a}_{10}| |x|+ |a_{01} -\widetilde{a}_{01}| |y-\frac{\sqrt{3}}{6}h|+ |\widetilde{a}_{20}||x|^2\\ \notag
&+& |\widetilde{a}_{11}||x||y-\frac{\sqrt{3}}{6}h|+ |\widetilde{a}_{02}||y-\frac{\sqrt{3}}{6}h|^2=O(h)+O(1)O(h)+O(1)O(h)=O(h).
\end{eqnarray}
Thus
\begin{equation*}
|f(x,y)-\widetilde{p}(x,y)|\leq |f(x,y)-p_1(x,y)|+|p_1(x,y)-\widetilde{p}(x,y)|= O(h^2)+O(h)=O(h), \ \forall (x,y) \in S_R,
\end{equation*}
which finishes the proof.
\end{proof}
\begin{remark}
We have defined the reconstruction in equilateral triangles using the harmonic mean of three values, but the reconstruction can be extended to whatever triangle by defining adequate weights depending on the specific form of the triangle, expressing the
coefficients in terms of weighted arithmetic means instead, and then following the same track as in the given example.
\end{remark}
The ideas expressed in the presented new reconstruction operator can be extrapolated to higher dimensions, and into other fields of numerical analysis.
To finish this section, we present a simple numerical example that reinforces the theoretical results. Given the following two functions of two variables $f(x,y)$ and $g(x,y),$
\begin{equation*}
f(x,y):=\sin{(x+y)}+20, \qquad g(x,y):=\left\{\begin{array}{cc}
\sin{(x+y)}+20, & y<-\sqrt{3}(x-\frac{5}{8}),\\
\cos{(x+y)}+200, & y\geq -\sqrt{3}(x-\frac{5}{8}),\\
\end{array} \right.
\end{equation*}
defined in the triangle $T$ of vertices
$A (-h,\frac{\sqrt{3}}{2}h),$
$C (0,-\frac{\sqrt{3}}{2}h),$ $E (h,\frac{\sqrt{3}}{2}h),$ with $h=0.005,$ we consider the linear reconstruction $p(x,y)$ given by (\ref{pollin}) and the nonlinear reconstruction $\widetilde{p}(x,y)$ given by (\ref{polPPH}) inside the triangle $S_R$ represented in Figure \ref{fig:tria3}, and also the same kind of reconstructions, but in the triangles $S_Y$ and $S_G$ with sides of length a half and a quarter of the length of the sides of the original triangle $S_R.$ Then, we measure the errors and the approximation order of both linear and associated nonlinear method in two scenarios, i.e., with the smooth function $f(x,y)$ and with the function $g(x,y)$ which contains a jump discontinuity along the straight line $y=-\sqrt{3}(x-\frac{5}{8}).$
In Figure \ref{fig:tria3}, we see the domain of the considered functions and the representation of the reconstructions attained in the triangle $S_R$ by both methods. One can easily observe how the linear method produces the expected Gibb phenomena around the jump discontinuity, while the nonlinear method seems to avoid it. This fact can also be appreciated in the Table \ref{tab:exp1}, where we have measured the committed errors for the two reconstructions inside the triangle $S_G$ when building the reconstructions for the three triangles $S_R,$ $S_Y,$ and $S_G$ respectively. We have also included the numerical approximation order computed from these errors, i.e., we have approximated the numerical order $p$ by using
\begin{equation*}
p\approx \log_2\dfrac{E_{S_R}}{E_{S_Y}}, \quad \textrm{and} \quad p\approx\log_2\dfrac{E_{S_Y}}{E_{S_G}},
\end{equation*}
where $E_{S_R},$ $E_{S_Y},$ and $E_{S_G}$ stand for the approximation errors in infinity norm inside the triangle $S_G,$ attained by the considered reconstruction operators, builded using the information relative to the indicated triangle. In the case of dealing with a smooth function, we see that the nonlinear method imitates the good behavior of its linear counterpart. This point can be appreciated as much in Figure \ref{fig:tria3} as in Table \ref{tab:exp2}.
We would like to remark the fact that the obtained numerical orders coincide with the expected according to Theorem \ref{teo:recolin} and Theorem \ref{teo:recoPPH}. Also, it is remarkable the fact that the linear method completely loses any approximation order in case of the jump discontinuity and produces Gibbs effects, while these drawbacks are avoided with the proposed nonlinear method, attaining at least a first order approximation.
\begin{figure} \label{fig:tria3}
\centerline{\includegraphics[width=7cm]{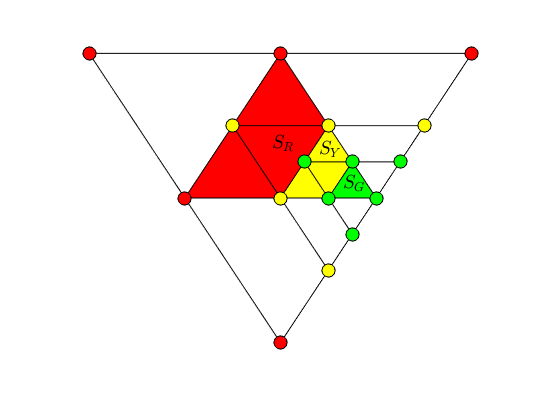}
\includegraphics[width=7cm]{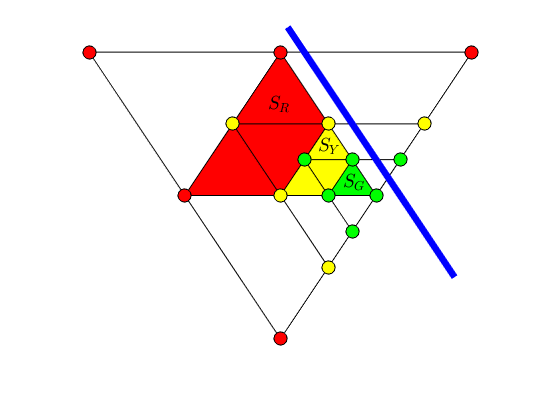}}
\centerline{\includegraphics[width=7cm]{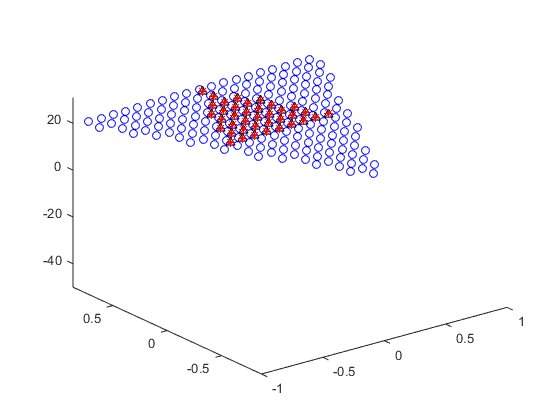}
\includegraphics[width=7cm]{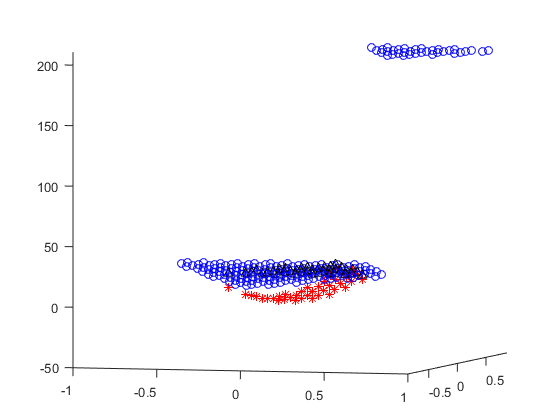}}
\caption{Top-left: Disposition of the considered domain to build the linear and nonlinear reconstructions inside the triangles $S_R,$ $S_Y,$ and $S_G.$ Top-right: Disposition of the considered domain affected by a jump discontinuity along the blue curve for which we build the linear and nonlinear reconstructions inside the triangles $S_R,$ $S_Y,$ and $S_G.$ Bottom-left: Obtained reconstructions, and comparison with the original smooth function $f(x,y)$ in the triangle $S_R.$ Bottom-right: Obtained reconstructions, and comparison with the original discontinuous function $g(x,y)$ in the triangle $S_R.$ With blue circles the original function, with red asterisks the linear reconstruction and with black triangles the new nonlinear reconstruction.}
\end{figure}
\begin{table}[!h]
\begin{center}
\begin{tabular}{|c|c|c|c|c|c|}
\hline
\multicolumn{3}{|c|}{$p(x,y)$} & \multicolumn{3}{c|}{$\widetilde{p}(x,y)$} \\
\hline
$Triangle$ & $||p(x,y)-f(x,y)||_\infty$ & $p$  & $Triangle$ & $||p(x,y)-f(x,y)||_\infty$ & $p$ \\
 \hline
$S_R$  & $16.9679$ & $-$  & $S_R$ & $7.0587\cdot 10^{-7}$ & $-$ \\
\hline
$S_Y$  & $22.6243$ & $-0.4151$  & $S_Y$ & $4.7168\cdot 10^{-7}$ & $0.5816$\\
\hline
$S_G$  & $22.6245$ & $-1.3679\cdot 10^{-5}$  & $S_G$ & $2.3545\cdot 10^{-7}$ & $1.0024$\\
\hline
\end{tabular}
\end{center}
\caption{Numerical approximation errors $||p(x,y)-g(x,y)||_{\infty}$ and $||\widetilde{p}(x,y)-g(x,y)||_{\infty}$  in infinity norm between the linear reconstruction and the original discontinuous function $g(x,y)$ and between the
 nonlinear reconstruction $\widetilde{p}(x,y)$ and the original discontinuous function $g(x,y)$ in the triangle $S_G$ for the cases of building the reconstructions inside the triangles $S_R,$ $S_Y$ and $S_G$ of decreasing side lengths. The approximation orders $p$ are also offered.} \label{tab:exp1}
\end{table}
\begin{table}[!h]
\begin{center}
\begin{tabular}{|c|c|c|c|c|c|}
\hline
\multicolumn{3}{|c|}{$p(x,y)$} & \multicolumn{3}{c|}{$\widetilde{p}(x,y)$} \\
\hline
$Triangle$ & $||p(x,y)-f(x,y)||_\infty$ & $p$  & $Triangle$ & $||p(x,y)-f(x,y)||_\infty$ & $p$ \\
 \hline
$S_R$  & $9.1721\cdot 10^{-9}$ & $-$  & $S_R$ & $8.9912\cdot 10^{-9}$ & $-$ \\
\hline
$S_Y$  & $1.6914\cdot 10^{-9}$ & $2.4390$  & $S_Y$ & $1.6687\cdot 10^{-9}$ & $2.4298$\\
\hline
$S_G$  & $2.1329\cdot 10^{-10}$ & $2.9874$  & $S_G$ & $2.1080\cdot 10^{-10}$ & $2.9848$\\
\hline
\end{tabular}
\end{center}
\caption{Numerical approximation errors $||p(x,y)-f(x,y)||_{\infty}$ and $||\widetilde{p}(x,y)-f(x,y)||_{\infty}$  in infinity norm between the linear reconstruction and the original smooth function $f(x,y)$ and between the
 nonlinear reconstruction $\widetilde{p}(x,y)$ and the original smooth function $f(x,y)$ in the triangle $S_G$ for the cases of building the reconstructions inside the triangles $S_R,$ $S_Y$ and $S_G$ of decreasing side lengths. The approximation orders $p$ are also offered.} \label{tab:exp2}
\end{table}

\section{Conclusions}\label{sec4}

In this article we have presented two relevant properties of the harmonic mean of three values that allow for new constructions of numerical methods, such as nonlinear reconstruction operators, subdivision and multiresolution
schemes, and solvers of hyperbolic conservation laws. We offer a clear and simple example on how to use these simple concepts to attain interesting and promising results in defining new reconstruction adapted methods. In fact, we have defined a new  reconstruction method for two dimensional functions which seems to avoid the Gibbs effects, and retains first order of approximation in the neighborhood of a jump discontinuity.

\clearpage

\end{document}